\documentclass[review]{siamart1116}

\usepackage{latexsym}
\usepackage{amsmath}
\usepackage{amssymb}
\usepackage{amsfonts}
\usepackage[latin1]{inputenc}
\usepackage{mathrsfs}
\usepackage{amsfonts}
\usepackage{graphicx}
\usepackage{colortbl,dcolumn}
\usepackage{amsmath}
\usepackage{amsfonts}
\usepackage{graphicx}
\usepackage{algorithmic}
\ifpdf
  \DeclareGraphicsExtensions{.eps,.pdf,.png,.jpg}
\else
  \DeclareGraphicsExtensions{.eps}
\fi

\newtheorem{tm}{Theorem}[section]
\newtheorem{rk}{Remark}[section]

\newtheorem{prop}{Proposition}[section]
\newtheorem{lm}{Lemma}[section]
\newtheorem{cor}{Corollary}[section]

\newcommand{\E}{\mathbb E}
\newcommand{\PP}{\mathbb P}
\newcommand{\N}{\mathbb N}
\newcommand{\R}{\mathbb R}

\newcommand{\LL}{\mathcal L}
\newcommand{\OO}{\mathcal O}
\newcommand{\HH}{\mathbb H}

\newcommand{\FFF}{\mathscr F}
\newcommand{\<}{\langle}
\renewcommand{\>}{\rangle}

\ifpdf
  \DeclareGraphicsExtensions{.eps,.pdf,.png,.jpg}
\else
  \DeclareGraphicsExtensions{.eps}
\fi

\newcommand{\TheTitle}{Wellposedness and regularity estimate for stochastic Cahn--Hilliard equation with unbounded noise diffusion} 
\newcommand{\TheAuthors}{Jianbo Cui and  Jialin Hong}

\headers{}{\TheAuthors}

\title{{\TheTitle}\thanks{This work was funded by National Natural Science Foundation of China (No.
91630312, No. 91530118, No.11021101 and No. 11290142).}}

\author{ Jianbo Cui 
\thanks{ School of Mathematics, Georgia Institute of Technology, Atlanta, GA 30332, USA
(\email{jcui82@gatech.edu} (corresponding author))
}
\and 
Jialin Hong 
\thanks{ LSEC, ICMSEC, 
Academy of Mathematics and Systems Science, Chinese Academy of Sciences, Beijing,  100190, China\qquad  School of Mathematical Science, University of Chinese Academy of Sciences, Beijing, 100049, China 
(\email{hjl@lsec.cc.ac.cn})}
}

\usepackage{amsopn}

\ifpdf
\hypersetup{
  pdftitle={\TheTitle},
  pdfauthor={\TheAuthors}
}
\fi

\usepackage{ulem}
\usepackage{subfigure}
\usepackage{graphicx}
\usepackage{enumerate}
\usepackage{amsmath,bm}
\usepackage{latexsym}
\usepackage{float}

\begin{document}

\maketitle

\begin{abstract}
In this article, we consider the one dimensional stochastic Cahn--Hilliard equation driven by multiplicative space-time white noise with diffusion coefficient of sublinear growth. 
By introducing the spectral Galerkin method, we first obtain the well-posedness of the approximated equation in finite dimension.  
Then with the help of the semigroup theory and the factorization method, the approximation processes is shown to possess many desirable properties.
Further, we show that the approximation process is strongly convergent in certain Banach space via the interpolation inequality and variational approach.  Finally, 
the global existence and 
regularity estimate of the unique solution process are proven by means of  
the strong convergence of the approximation process.
\end{abstract}


\begin{AMS}
{60H15}, {60H35, 35R60.}
\end{AMS}

\section{Introduction}

In this article, we consider the following stochastic  Cahn--Hilliard equation with multiplicative space-time white noise 
\begin{align}\label{spde}
dX(t)+A(AX(t)+F(X(t)))dt&=G(X(t))dW(t),\quad t\in (0,T],\\\nonumber
X(0)&=X_0.
\end{align}
Here $0<T<\infty$, $H:=L^2(\mathcal O)$ with $\mathcal O=(0,L), L>0$, 
$-A:D(A)\subset H \to H$ is the  Laplacian operator under homogenous Dirichlet or Neumman boundary condition, and  $\{W(t)\}_{t\ge 0}$ is a generalized  Wiener process on a filtered probability space $(\Omega,\mathcal F,\{\mathcal F_t\}_{t\ge 0},\PP)$. 
The nonlinearity $F$ is assumed to be the Nemytskii operator of $f'$, where $f$ is a polynomial of degree 4, i.e., $c_4\xi^4+c_3\xi^3+c_2\xi^2+c_1\xi+c_0$ with $c_i\in \R$, $i=0,\cdots,4$, $c_4>0$.
A typical example is  the double well potential $f=\frac 14(\xi^2-1)^2$.  For more general drift nonlinearities, we refer to \cite{DG11} and references therein.
The diffusion coefficient $G$ is assumed to be the
Nemytskii operator of $g$, where $g$ is a global Lipschitz function with the sublinear growth condition $|g(\xi)|\le C(1+|\xi|^{\alpha}), \alpha<1$. When $G=I$, Eq. \eqref{spde} corresponds to the stochastic Cahn--Hilliard--Cook equation. This equation is used to describe the complicated phase separation and coarsening phenomena in a melted alloy that is quenched to a temperature at which only two different concentration phases can exist stably (see e.g. \cite{ABK12,CH58,NS84}). The physical importance of the Dirichlet problem was pointed out to us by M. E. Gurtin: it governs the propagation of a solidification front into an ambient medium which is at rest relative to the front (see e.g. \cite{DN91}). 

The existence and uniqueness of the solution to Eq. \eqref{spde} have already been  
proven by \cite{DD96} in the case of $G=I$ for the space dimension $d=1$.
Moreover, if $G=I$ but $d\ge2$, the driving noise should be more regular than the space-time white noise. When $G$ is a bounded diffusion coefficient, the authors in \cite{Car01} obtain the 
global existence and path regularity of the solution in $d=1$, and the local existence of the solution in higher dimension $d=2,3$. Recently, the authors in \cite{AKM16} extend the results on the local existence and uniqueness of the solution in the case that $|g(\xi)|\le C(1+|\xi|^{\alpha})$, $\alpha\in (0,1]$, $d\le 3$. Meanwhile, the global existence of the solution is achieved under the restriction that $\alpha<\frac 13$, $d=1$. However, for the global existence of the solution, it is still unknown whether the sublinear growth condition $\alpha <\frac 13$ could be 
extended to the general sublinear growth condition, i.e.,
$|g(\xi)|\le C(1+|\xi|^{\alpha}), \alpha \in (0,1)$, which is one main motivation of this article.

To study such problem, 
our strategy is different from that in the existing literature 
(see e.g. \cite{Car01,AKM16}). 
Instead of introducing an appropriated cut-off SPDE, we  
firstly use the spectral Galerkin method to discretize Eq. \eqref{spde} and get the spectral Galerkin approximation 
\begin{align}\label{sga}
dX^N(t)+A(AX^N(t)+P^NF(X^N(t)))dt&=P^NG(X^N(t))dW(t),\quad t\in (0,T]\\\nonumber
X^N(0)&=P^NX_0,
\end{align}
where $N\in \N^+$.
Then by making use of the factorization formula and the equivalent random form of the semi-discrete equation, we show the well-posedness of  the semi-discrete equation \eqref{sga}, as well as its uniform a priori estimate and regularity estimate. Furthermore, we show that the limit of the solution of the spectral Galerkin method exists globally and  is the unique mild solution of Eq. \eqref{spde}.
As a consequence, the exponential integrability property, the optimal temporal and spatial regularity estimates of the exact solution are proven. 
Meanwhile,  with the help of the Sobolev interpolation equality and the smoothing effect of the semigroup $S(t):=e^{-A^2t}$, the sharp spatial strong convergence rate of the spectral Galerkin method is established under homogenous Dirichlet boundary condition. To the best of our knowledge, this is not only the new result on the global  existence and regularity estimate of the solution, but also the first result on the strong convergence rate of  numerical approximation for the stochastic Cahn--Hilliard equation driven by multiplicative space-time white noise.

 The rest of this article is organized as follows. In Section \ref{sec-2} the setting and assumptions used are formulated. In Section \ref{sec-pri}, we prove several uniform a priori estimates and regularity estimates of the spatial spectral Galerkin method. The strong convergence analysis of the spatial spectral Galerkin method is presented in Section \ref{sec-str}.
Our main result which states existence, uniqueness and regularity of solutions of Eq. \eqref{spde} with nonlinear multiplicative noise is presented in Section \ref{sec-wel}.

\section{Preliminaries}\label{sec-2} 
In this section, we present some preliminaries and notations, as well as the assumptions on Eq. \eqref{spde}.

Given two separable Hilbert spaces $(\mathcal H, \|\cdot \|_{\mathcal H})$ and $(\widetilde  H,\|\cdot \|_{\widetilde H})$, 
$\LL(\mathcal H, \widetilde H)$ and $\LL_1(\mathcal H, \widetilde H)$  are the Banach spaces of all linear bounded operators 
and  the nuclear operators from $\mathcal H$ to $\widetilde H$, respectively. 
The trace of an operator $\mathcal T\in \LL_1(\mathcal H)$
is $tr[\mathcal T]=\sum_{k\in \N}\<\mathcal Tf_k,f_k\>_{\mathcal H}$, where $\{f_k\}_{k\in \N}$ ($\N=\{0,1,2,\cdots\}$) is any orthonormal basis of $\mathcal H$.
In particular, if $\mathcal T\ge 0$, $tr[\mathcal T]=\|\mathcal T\|_{\mathcal L_1}$.
Denote by $\LL_2(\mathcal H,\widetilde H)$ the space 
of Hilbert--Schmidt operators from $\mathcal H$ into $\widetilde H$, equipped with the usual norm given by  $\|\cdot\|_{\LL_2(\mathcal H,\widetilde H)}=(\sum_{k\in \N}\|\cdot f_k\|^2_{\widetilde H})^{\frac{1}{2}}$.
The following useful property and inequality hold 
\begin{align}\label{tr}
&\|\mathcal S \mathcal T\|_{\mathcal L_2(\mathcal H,\widetilde H)}\le \|\mathcal S\|_{\mathcal L_2(\mathcal H,\widetilde H)}\|\mathcal T\|_{\mathcal L(\mathcal H)},\quad \mathcal T \in \mathcal L(\mathcal H), \;\;\mathcal S\in\mathcal L_2(\mathcal H,\widetilde H),\\\nonumber
&tr[\mathcal Q]=\|\mathcal Q^{\frac 12}\|^2_{L_2(\mathcal H)}=\|\mathcal T\|^2_{ \mathcal L_2(\widetilde H,\mathcal H)},\quad \mathcal Q=\mathcal T \mathcal T^{*},\;\; \mathcal T\in \mathcal L_2(\widetilde H,\mathcal H),
\end{align}
where $\mathcal T^*$ is the adjoint operator of $\mathcal T$.

Given a Banach space $(\mathcal E,\|\cdot\|_{\mathcal E})$ and $T\in \mathcal L(\mathcal H,\mathcal E)$, we denote by $\gamma( \mathcal H, \mathcal E)$ the space of $\gamma$-radonifying operators endowed with the norm
$\|\mathcal T\|_{\gamma(\mathcal  H, \mathcal E)}=(\widetilde \E\|\sum_{k\in\N}\gamma_k \mathcal Tf_k \|^2_{\mathcal E})^{\frac 12}$,
where $(\gamma_k)_{k\in\N}$ is a Rademacher sequence on a
probability space $(\widetilde \Omega,\widetilde \FFF, \widetilde \PP)$.
For convenience, let   $L^q=L^q(\OO)$, $2\le q<\infty$ equipped with the usual inner product and norm.
We also need the following Burkerholder inequality (see e.g. \cite{VVW08}), 
\begin{align}\label{Burk}
\left\|\sup_{t\in [0,T]}\Big\|\int_0^t \phi(r)d\widetilde W(r)\Big\|_{L^q}\right\|_{L^p(\Omega)}
&\le 
C_{p,q}\|\phi\|_{L^p(\Omega;  L^2([0,T]; \gamma(H;L^q))}\\\nonumber 
&\le  C_{p,q}\Big(\E\Big(\int_0^T\Big\|\sum_{k\in \N} (\phi(t) e_k)^2\Big\|_{L^{
\frac q2}}dt\Big)^{\frac p2}\Big)^{\frac 1p}, 
 \end{align}
 where $\widetilde W$ is the $H$-valued cylindrical 
 Wiener process and  $\{e_k\}_{k\in \N}$ is any orthonormal basis of $H$. 

Next, we introduce some assumptions and spaces associated with $ A$.
 We denote by 
$H^k:=H^k(\mathcal O)$ the standard Sobolev space and $E:=\mathcal C(\mathcal O)$. 
For convenience, we mainly focus on the well-posedness and numerical approximation for Eq. \eqref{spde} under homogenous Dirichlet boundary condition. We would like to mention that the approach for proving the global existence of the unique solution is also available for Eq. \eqref{spde} under homogenous Neumman  boundary condition. 
Denote $A=-\Delta $ the Dirichlet Laplacian operator with 
$$D(A)=\left\{v\in H^2(\mathcal O): v =0\;\; \text{on} \;\; \partial \mathcal O\right\}.$$
It is known that $A$ is a positive definite, self-adjoint and unbounded linear operator on $H$. Thus there exists an orthonormal 
eigensystem $\{(\lambda_j,e_j)\}_{j\in \N}$ such that $0<\lambda_1\le \cdots \le \lambda_j\le \cdots$ with $\lambda_j \sim j^{2}$ and $\sup_{j\in \N^+}\|e_j\|_E<\infty$.
We  define 
$\HH^{\alpha}$, $\alpha\in \R$ as the space of the series $v:=\sum_{j=1}^{\infty}v_je_j$, $v_j\in \R$, such that
$\|v\|_{\HH^{\alpha}}:=(\sum_{j=1}^{\infty}\lambda_j^{\alpha}v_j^2)^{\frac 12}<\infty$. Equipped with the norm $\|\cdot\|_{\HH^{\alpha}}$ and corresponding inner product, the Hilbert space $\HH^{\alpha}$ equals $D(A^{\frac \alpha 2})$. It is obvious that $H=\HH$. We denote $\|\cdot\|=\|\cdot\|_{\HH}.$ The following smoothing effect of the analytical semigroup $S(t)=e^{-t A^2},t>0$ (see e.g. \cite{EN00}), 
\begin{align}\label{smo-eff}
\|A^{\beta}S(t)v\|&\le Ct^{-\frac \beta 2}\|v\|, \;\beta>0, \; v\in \mathbb H
\end{align}
and the contractivity property of $S(t)$ (see e.g. \cite[Appendix B]{PZ07}),
\begin{align}\label{con-pro}
\|S(t)v\|_{L^q}&\le Ct^{-\frac 14(\frac 1p-\frac 1q)} \|v\|_{L^p},\; 1\le p\le q<\infty ,\;  v\in L^p,\\\nonumber 
\|S(t)v\|_{E}&\le Ct^{-\frac 1{4p}} \|v\|_{L^p}, v\in L^p,
\end{align}
will be used frequently.
Throughout this article, the Wiener process $W$ is assumed to be the $H$-valued cylindrical Wiener process, which implies that for any 
$\gamma\in (0,\frac 32) $,  $\|A^{\frac {\gamma-2}2}Q^{\frac 12}\|_{\LL_2(\HH)}<\infty$.  We denote by $C$ a generic constant which may depend on several parameters but never on the projection parameter $N$ and may change from occurrence to occurrence. We also remark that the approach for proving the global existence of the unique solution 
is available for the cases of higher dimension and more regular $Q$-Wiener process.

\section{A priori estimate and regularity estimate of the spectral Galerkin method}
\label{sec-pri}

In this section, we give the a priori estimate and regularity estimate of the solution of Eq. \eqref{sga}. 
Notice that Eq. \eqref{sga} is equivalent to the following random PDE and the equation of the discrete stochastic convolution $Z^N$, 
\begin{align}\label{rand-y}
&dY^N(t)+A(AY^N(t)+P^NF(Y^N(t)+Z^N(T)))dt=0,\;
Y^N(0)=P^NX_0,\\\label{sto-con}
&dZ^N(t)+A^2Z^N(t)dt=P^NG(Y^N(t)+Z^N(t))dW(t),\;
Z^N(0)=0.
\end{align}
The above decomposition is inspired by \cite{Cer03} where the authors use similar decomposition to show the well-posedness of stochastic reaction-diffusion systems.
In the following, we present the  
a priori and regularity estimates of $Z^N$ and $Y^N$.

\begin{lm}\label{lm-xn}
Let $X_0\in \HH$, $T>0$ and $q\ge1$.
There exists a unique solution  $X^N$  of Eq. \eqref{sga}  satisfying 
\begin{align}\label{p-mom}
\sup_{t\in [0,T]}\E\Big[ \big\|X^N(t)\big\|_{\HH^{-1}}^q\Big]
&\le C(X_0,T,q),
\end{align}
where $ C(X_0,T,q)$ is a positive constant.
\end{lm}
\begin{proof}
Thanks to the fact all the norms in finite dimensional normed linear spaces are equivalent, the norm $\|\cdot\|:=\|\cdot\|_{\HH}$ and $\|\cdot\|_{\HH^{-1}}$ in $P^N(\HH)$ are equivalent up to constants depending on $N$.
The existence of a unique strong solution for Eq. \eqref{sga} in $\HH^{-1}$ can be obtained by the arguments in \cite[Chapter 3]{PR07}. However, the moment bound of the exact solution will depend on $N$ by this procedure.
To  prove \eqref{p-mom}, we need to find a proper Lyapunov functional and to derive the a priori estimate independent of $N$.
According to Eq. \eqref{rand-y}, by using the chain rule and integration by parts, we have for any $t\le T$,
\begin{align*}
\|Y^N(t)\|_{\HH^{-1}}^2
&=
\|Y^N(0)\|_{\HH^{-1}}^2
-2\int_0^t\<\nabla Y^N(s), \nabla Y^N(s)\>ds\\
&\quad-
2\int_0^t\<F(Y^N(s)+Z^N(s)),Y^N(s)\>ds\\
&=\|Y^N(0)\|_{\HH^{-1}}^2
-2\int_0^t\|\nabla Y^N(s)\|^2 ds\\
&\quad-
2\int_0^t\<F(Y^N(s)+Z^N(s)),Y^N(s)\>ds.
\end{align*}
The expression of $F$ and Young inequality implies that 
\begin{align}\label{pri-h1}
&\|Y^N(t)\|_{\HH^{-1}}^2+2\int_0^t\|\nabla Y^N(s)\|^2ds+8(c_4-\epsilon)\int_0^t\|Y^N(s)\|^4_{L^4}ds\\\nonumber
&\le \|Y^N(0)\|_{\HH^{-1}}^2
+C(\epsilon)\int_0^t(1+\|Z^N(s)\|_{L^4}^4)ds.
\end{align}
Thus it suffices to deduce the a priori estimate of $\int_0^t \|Z^N(s)\|_{L^4}^4ds$.
From the mild form of $Z^N$, the H\"older inequality, the Burkholder inequality and the contractivity of $S(\cdot)$ \eqref{con-pro}, 
it follows that for $p\ge 2$ and $q\ge 4$, 
\begin{align*}
&\E[\|Z^N(s)\|_{L^{p}}^q]\\
&=\E\Big[\Big\|\int_0^sS(s-r)P^NG(Y^N(r)+Z^N(r))dW(r)\Big\|_{L^p}^q\Big]\\
&\le 
C\E\Big[\Big(\int_0^s\Big\|S(s-r)P^NG(Y^N(r)+Z^N(r))\Big\|_{\gamma(\HH,L^p)}^2dr\Big)^{\frac q2}\Big]
\\
&\le 
C\E\Big[\Big(\int_0^s\sum_{k=1}^{\infty}\Big\|S(s-r)P^N(G(Y^N(r)+Z^N(r))e_k)\Big\|_{L^p}^2dr\Big)^{\frac q2}\Big]\\
&\le 
C\E\Big[\Big(\int_0^s(s-r)^{-\frac 12(\frac 12-\frac 1p)}\sum_{k=1}^{\infty}\Big\|S(\frac {s-r}2)P^N(G(Y^N(r)+Z^N(r))e_k)\Big\|^2dr\Big)^{\frac q2}\Big].
\end{align*}
The Paserval equality and the sublinear growth of $G$ yield that
\begin{align*}
&\E[\|Z^N(s)\|_{L^{p}}^q]\\
&\le
C\E\Big[\Big(\int_0^s(s-r)^{-\frac 12(\frac 12-\frac 1p)}\sum_{j,k=1}^{\infty}\Big\<G(Y^N(r)+Z^N(r))e_k,e^{-\frac 12\lambda_j^2(s-r)}e_j\Big\>^2dr\Big)^{\frac q2}\Big]\\
&=C\E\Big[\Big(\int_0^s(s-r)^{-\frac 12(\frac 12-\frac 1p)}\sum_{j=1}^{\infty}e^{-\lambda_j^2(s-r)}\|G(Y^N(r)+Z^N(r))e_j\|^2dr\Big)^{\frac q2}\Big]\\
&\le C\E\Big[\Big(\int_0^s(s-r)^{-\frac 12(\frac 12-\frac 1p)}\sum_{j=1}^{\infty}e^{-\lambda_j^2(s-r)}\|G(Y^N(r)+Z^N(r))\|^2dr\Big)^{\frac q2}\Big]\\
&\le C\E\Big[\Big(\int_0^s(s-r)^{-\frac 12+\frac 1{2p}}\|G(Y^N(r)+Z^N(r))\|^2dr\Big)^{\frac q2}\Big]\\
&\le C\E\Big[\Big(\int_0^s(s-r)^{\frac {-p+1}{2p}}(1+\|Y^N(r)\|^{2\alpha}+\|Z^N(r))\|^{2\alpha})dr\Big)^{\frac q2}\Big]\\
&\le C(\int_0^s(s-r)^{\frac {-p+1}{p}}dr)^{\frac q4} \E\Big[\Big(1+\int_0^{s}(\|Y^N(r)\|^{4\alpha}+\|Z^N(r))\|^{4\alpha})dr\Big)^{\frac q4}\Big].
\end{align*}
Using the Young inequality, we obtain for $0\le s\le t$,
\begin{align*}
\E[\|Z^N(s)\|_{L^{p}}^q]
&\le  Cs^{\frac q{4p}} \Big(1+\E[(\int_0^s\|Y^N(r)\|^{4 \alpha}dr)^{\frac q4}]+\int_0^s\E[\|Z^N(r)\|^{q}]dr\Big)
\\ 
&\le Cs^{\frac q{4p}} \Big(1+\E[(\int_0^s\|Y^N(r)\|_{L^p}^{4 \alpha}dr)^{\frac q4}]+\int_0^s\E[\|Z^N(r)\|_{L^p}^{q}]dr\Big).
\end{align*}
Since the moment bound of $Z^N$ and $Y^N$ are finite depending on $N$, we can apply the Gronwall's inequality and get that for  $0\le s\le T$,
\begin{align}\label{pri-zn}
\E[\|Z^N(s)\|_{L^{p}}^q]
&\le C(T)\Big(1+\E[(\int_0^s\|Y^N(r)\|_{L^p}^{4 \alpha}dr)^{\frac q4}]\Big).
\end{align}

Now taking $k$th moment, $k\in \N^+$ on \eqref{pri-h1} and letting $p=4$, $q=4k$, we have 
\begin{align*}
&\E\Big[(\int_0^t\|Y^N(s)\|^4_{L^4}ds)^k\Big]\\
&\le C\|Y^N(0)\|_{\HH^{-1}}^{2k}
+C(\epsilon)\int_0^t(1+\E[\|Z^N(s)\|_{L^4}^{4k}])ds\\
&\le C\|Y^N(0)\|_{\HH^{-1}}^{2k}+C(\epsilon,T)\Big(C(\epsilon_1)+\epsilon_1 \int_0^t\E[(\int_0^s\|Y^N(r)\|_{L^4}^{4}dr)^{k}]ds\Big),
\end{align*}
where $\epsilon_1>0$ is a small number such that 
$C(\epsilon,T)\epsilon_1T<\frac 12$. 
The above estimation leads to 
\begin{align*}
\E\Big[(\int_0^t\|Y^N(s)\|^4_{L^4}ds)^k\Big]
&\le  C\|Y^N(0)\|_{\HH^{-1}}^{2k}+C(k,\epsilon,\epsilon_1,T),
\end{align*}
which in turns yields that for $k\in \N^+$,
\begin{align}\label{pri-h-1}
&\E\Big[\|Y^N(t)\|_{\HH^{-1}}^{2k}\Big]+\E\Big[(\int_0^t\|\nabla Y^N(s)\|^2ds)^k\Big]+\E\Big[(\int_0^t\|Y^N(s)\|^4_{L^4}ds)^{k}\Big]\\\nonumber
&\le C(X_0,T,k).
\end{align}
Based on the a priori estimates of $Z^N$  and $Y^N$ in $L^{p}$ and $\HH^{-1}$, respectively, we complete the proof via the H\"older inequality.
\end{proof}

\begin{lm}\label{cor-zn}
Let  $X_0\in \HH$, $T>0$ and $q\ge1$.
There exists  a positive constant $ C(X_0,T,q)$ such that 
\begin{align}\label{pri-sup-zn}
\E\Big[\sup_{t\in [0,T]} \big\|Z^N(t)\big\|_{E}^q\Big]
&\le C(X_0,T,q).
\end{align}
\end{lm}
\begin{proof}
By using the factorization formula in \cite[Chapter 5]{Dap92}, we have that for $\alpha_1>\frac 1p+\gamma$, $p>1$, $\gamma=\frac 18$,
\begin{align*}
\E\Big[\sup_{s\in [0,T]}\|Z^N(s)\|_{E}^q\Big]
&\le C(q,T)\E\Big[\|Y_{\alpha_1,N}\|_{L^p(0,T; \HH)}^q\Big],
\end{align*}
where $Y_{\alpha_1,N}(s)=\int_0^s(s-r)^{-\alpha_1}S(s-r)
P^NG(Y^N(r)+Z^N(r))dW(r)$.
Thus it suffices to estimate $\E\Big[\|Y_{\alpha_1,N}\|_{L^p(0,T; \HH)}^q\Big]$. 
From the H\"older and Burkholder inequalities,  it follows that for $q\ge \max(p,2)$,
\begin{align*}
&\E\Big[\|Y_{\alpha_1,N}\|_{L^p(0,T; \HH)}^q\Big]\\
&=
\E\Big[ \Big(\int_0^T\Big\|\int_0^s(s-r)^{-\alpha_1}S(s-r)
P^NG(Y^N(r)+Z^N(r))dW(r)\Big\|^pds\Big)^{\frac qp}\Big]\\
&\le C(T,q)\int_0^T\E\Big[ \Big\|\int_0^s(s-r)^{-\alpha_1}S(s-r)
P^NG(Y^N(r)+Z^N(r))dW(r)\Big\|^q\Big]ds\\
&\le C(T,q)\int_0^T\E\Big[ \Big(\int_0^s(s-r)^{-2\alpha_1}\sum_{i\in \N^+}\Big\|S(s-r)
P^NG(Y^N(r)+Z^N(r))e_i\Big\|^2dr\Big)^{\frac q2}ds\Big]\\
&\le C(T,q)\int_0^T\E\Big[ \Big(\int_0^s(s-r)^{-(2\alpha_1+\frac 14)}(1+\|Y^N(r)\|^{2\alpha}+\|Z^N(r)\|^{2\alpha})dr\Big)^{\frac q2}ds\Big].
\end{align*}
Since $\alpha<1$, one can choose a positive number $l>2$ and a  large enough number $p$ 
 such that
$2\alpha l<4$ and $(2\alpha_1+\frac 14)\frac l{l-1}<1$.
Then by using a priori estimates \eqref{pri-zn} and \eqref{pri-h-1}, we obtain 
\begin{align*}
\E\Big[\|Y_{\alpha,N}\|_{L^p(0,T; H)}^q\Big]
&\le C(T,q,\alpha)
\int_0^T(\int_0^s(s-r)^{-(2\alpha_1+\frac 14)\frac l{l-1}}dr)^{\frac {q(l-1)}{2l}}\\
&\qquad \times
\E\Big[\Big(\int_0^s(1+\|Y^N(r)\|^{2\alpha l}+\|Z^N(r)\|^{2\alpha l})dr\Big)^{\frac q{2l}}\Big]ds\\
&\le  C(T,q,\alpha,X_0), 
\end{align*}
which implies that 
\begin{align*}
\E\Big[\sup_{s\in [0,T]}\|Z^N(s)\|_{E}^q\Big]
&\le C(T,q,\alpha,X_0).
\end{align*}
\end{proof}

\begin{cor}
Let $X_0\in \HH$, $T>0$ and $q\ge1$.
Then the solution  $X^N$  of Eq. \eqref{sga}  satisfies
\begin{align}\label{p-mom-h-1}
\E\Big[\sup_{t\in [0,T]} \big\|X^N(t)\big\|_{\HH^{-1}}^q\Big]
&\le C(X_0,T,q),
\end{align}
where $ C(X_0,T,q)$ is a positive constant.
\end{cor}
\begin{proof}
Similar arguments in the proof of \eqref{pri-h-1} yield that 
for any $k\ge1$,
$$\E\Big[\sup_{t\in[0,T]}\|Y^N(t)\|_{\HH^{-1}}^{2k}\Big]\le C(X_0,T,k).$$ Combining this estimate with Lemma \ref{cor-zn}, we complete the proof.
\end{proof}
Thanks to the above a priori estimates of $Y^N$ and $Z^N$, we are now in a position to deduce the a priori estimate of $X^N$ in $\HH$.  
\begin{lm}\label{pri-xn}
Let $X_0\in \HH$, $T>0$ and $q\ge1$.
There exists  a positive constant $ C(X_0,T,q)$ such that 
\begin{align}\label{p-mom-xn}
\E\Big[\sup_{t\in [0,T]} \big\|X^N(t)\big\|_{\HH}^q\Big]
&\le C(X_0,T,q).
\end{align}
\end{lm}
\begin{proof}
By applying the integration by parts and the dissipativity  of $-F$,
we obtain 
\begin{align*}
&\|Y^N(t)\|^2+(2-\epsilon)\int_0^t\|(-A) Y^N(s)\|^2ds\\
&\le \|X_0^N\|^2
+C\int_0^t(\|Z^N(s)\|_E^2+1)\|Y^N(s)\|_{L^4}^4ds\\
&\quad+
C\int_0^t
(1+\|\nabla Y^N(s)\|^2+\|Z^N(s)\|_{L^8}^8)ds.
\end{align*}
Taking the $p$th moment and using the a priori estimates \eqref{pri-h-1} and \eqref{pri-sup-zn}, we have that for $p\ge1$,
\begin{align*}
&\E\Big[\sup_{t\in[0,T]}\|Y^N(t)\|^{2p}\Big]+\E\Big[\int_0^T\|(-A)Y^N(s)\|^{2p}ds\Big]
\\
&\le C(p,T)\Big(\E\Big[\|X_0^N\|^{2p}\Big]
+
\E\Big[(1+\sup_{s\in[0,T]}\|Z^N(s)\|_E^{2p})(\int_0^T\|Y^N(s)\|^{4}_{L^4}ds)^p\Big]\\
&\quad+
\E \Big[\Big(\int_0^T
(1+\|\nabla Y^N(s)\|^2+\|Z^N(s)\|_{L^8}^8)ds\Big)^p\Big]
\Big)\\
&\le  C(T,X^N_0,p),
\end{align*}
which, together with \eqref{pri-sup-zn} and the H\"older inequality, completes the proof.

\end{proof}

Based on the a priori estimate of $\|X^N\|$, we are in a position to deduce the regularity estimate of $X^N$. Before that, we first give the regularity estimate of $Z^N$.

\begin{lm}\label{reg-zn}
Let $X_0\in \HH$, $q\ge 1$ and $\gamma\in (0,\frac 32)$.
Then the discrete stochastic convolution $Z^N$ satisfies 
\begin{align}\label{reg-z}
\E\Big[\sup_{t\in [0,T]}\big\|Z^N(t)\big\|_{\HH^{\gamma}}^q\Big]
&\le C(X_0,T,q)
\end{align}
for a positive constant $ C(X_0,T,q)$.
\end{lm}

\begin{proof}
By the factorization method, we have 
for $\alpha_1>\frac 1p+\beta$, $p>1$, $\beta=\frac \gamma 4$,
\begin{align*}
\E\Big[\sup_{s\in [0,T]}\|Z^N(s)\|_{\HH^{\gamma}}^q\Big]
&\le C(T,q)\E\Big[\|Y_{\alpha_1,N}\|_{L^p(0,T; \HH)}^q\Big],
\end{align*}
where $Y_{\alpha_1,N}(s)=\int_0^s(s-r)^{-\alpha_1}S(s-r)
P^NG(Y^N(r)+Z^N(r))dW(r)$.
From the H\"older and Burkholder inequalities, the estimates  \eqref{pri-h-1} and  \eqref{p-mom-xn},  it follows that for $q\ge \max(p,2)$,
\begin{align*}
&\E\Big[\|Y_{\alpha_1,N}\|_{L^p(0,T; \HH)}^q\Big]\\
&\le C(T,q)\int_0^T\E\Big[ \Big(\int_0^s(s-r)^{-2\alpha_1}\sum_{i\in \N^+}\Big\|S(s-r)
P^NG(Y^N(r)+Z^N(r))e_i\Big\|^2dr\Big)^{\frac q2}ds\Big]\\
&\le C(T,q)\int_0^T\E\Big[ \Big(\int_0^s(s-r)^{-(2\alpha_1+\frac 14)}(1+\|Y^N(r)\|^{2\alpha}+\|Z^N(r)\|^{2\alpha})dr\Big)^{\frac q2}ds\Big]\\
&\le C(T,q)\Big(1+\E\Big[ \sup_{r\in [0,T]}\|Y^N(r)\|^{2\alpha q}\Big]+\E\Big[\sup_{r\in [0,T]}\|Z^N(r)\|^{2\alpha q}\Big]\Big)\\
&\quad \times \int_0^T(\int_0^s(s-r)^{-(2\alpha_1+\frac 14)}dr)^{\frac q2}ds\\
&\le C(X_0,T,q,\alpha)
\int_0^T(\int_0^s(s-r)^{-(2\alpha_1+\frac 14)}dr)^{\frac q2}ds.
\end{align*}
Since $\gamma <\frac 32$, one can choose a positive  a  large enough number $p$ 
 such that
$\frac 2p+\frac \gamma 2+\frac 14<1$ and $2\alpha_1+\frac 14<1$.
Thus we obtain 
\begin{align*}
\E\Big[\|Y_{\alpha_1,N}\|_{L^p(0,T; H)}^q\Big]
\le C(X_0,T,q,\gamma),
\end{align*}
which implies that for any $q\in \N^+$, 
\begin{align*}
\E\Big[\sup_{s\in [0,T]}\|Z^N(s)\|_{\HH^{\gamma}}^q\Big]
\le
C(X_0,T,q,\gamma). 
\end{align*}
\end{proof}
Next, we deduce the following uniform regularity estimate of 
$X^N$.
\begin{prop}\label{prop-spa}
Let $ X_0\in \HH^{\gamma}$,  $\gamma \in [1,\frac 32)$, $T>0$, $q\ge1$ and $N\in \N^+$.
Then the unique mild solution  $X^N$  of Eq. \eqref{sga}  satisfies 
\begin{align}\label{reg-x}
\E\Big[\sup_{t\in [0,T]}\big\|X^N(t)\big\|_{\HH^{\gamma}}^q\Big]
&\le C(X_0,T,q)
\end{align}
for a positive constant $C(X_0,T,q)$.
\end{prop}
\begin{proof}
Due to \eqref{reg-z},
it suffices to give the regularity estimate for $Y^N$.
Before that, we give the following estimate of $\|Y^N(t)\|_{L^6}$.
The Sobolev embedding theorem, the contractivity \eqref{con-pro} from $L^6$ to $L^2$, the smoothing effect \eqref{smo-eff}, and the Gagliardo--Nirenberg inequality yield that 
\begin{align*}
&\|Y^N(t)\|_{L^6}\\
&\le \|S(t)X_0^N\|_{L^6}
+\int_0^t \|S(\frac {t-s} 2)(S(\frac {t-s} 2)A)P^NF(Y^N(s)+Z^N(s))\|_{L^6}ds\\
&\le  C\|X_0^N\|_{L^6}
+C\int_0^t(t-s)^{-\frac 12}\|S(\frac {t-s} 2)A\| \|F(Y^N(s)+Z^N(s))\|ds\\
&\le C\|X_0^N\|_{L^6}
+C\int_0^t(t-s)^{-\frac 7{12}}
\Big(1+\|Z^N(s)\|_{L^{6}}^{3}+\| Y^N(s)\|_{L^{6}}^{3}\Big)ds\\
&\le C\|X_0^N\|_{\HH^1}
+C\int_0^t(t-s)^{-\frac 7{12}}
\Big(1+\|Z^N(s)\|_{L^{6}}^{3}+
\|AY^N(s)\|^{\frac 12}\|Y^N(s)\|^{\frac 52}\Big)ds.
\end{align*}
From the H\"older inequality, the estimates \eqref{pri-h-1}, \eqref{p-mom-xn} and \eqref{reg-z}, 
it follows that for any $q\ge 1$,
\begin{align*}
\E\Big[\sup_{t\in [0,T]}\|Y^N(t)\|_{L^6}^q\Big]
&\le C(p)\|X_0^N\|_{\HH^1}^q
+C(p)\E\Big[\Big(\int_0^T\|AY^N(s)\|^2ds\Big)^{\frac q4}\Big]
\\
&\quad+C(p)(\int_0^T(t-s)^{-\frac 7{9}}ds)^{\frac {3q}4}\E\Big[\sup_{s\in[0,T]}\|Y^N(s)\|^{\frac {10q}3}\Big]\\
&\quad+C(p)(\int_0^T(t-s)^{-\frac 7{12}}ds)^q
\Big(1+\E\Big[\sup_{s\in [0,T]}\|Z^N(s)\|_E^{3q}\Big]\Big)\\
&\le C(X_0,T,q).
\end{align*}
The mild form of $Y^N(t)$ and  \eqref{smo-eff} lead to 
\begin{align*}
\|Y^N(t)\|_{\mathbb \HH^{\gamma}}
&\le \|e^{-A^2t}X_0^N\|_{\HH^{\gamma}}
+\int_0^t\big\|e^{-A^2(t-s)}AF(Y^N(s)+Z^N(s))\big\|_{\HH^{\gamma}}ds\\
&\le C\|X_0^N\|_{\HH^{\gamma}}
+C\int_0^t(t-s)^{-\frac 12}\big\|e^{-\frac 12A^2(t-s)}F(Y^N(s)+Z^N(s))\big\|_{\HH^{\gamma}}ds\\
&\le C\|X_0\|_{\HH^{\gamma}}
+C\int_0^t(t-s)^{-\frac 12-\frac {\gamma}4}\big(1+\|Y^N(s)\|_{L^6}^3+\|Z^N(s)\|_{L^6}^3\big)ds.
\end{align*}
By taking $q$th moment and making use of the a priori estimates of  $\|Y^N\|_{L^6}$ and $\|Z^N\|_{\HH^{\gamma}}$, we finish the proof. 
\end{proof}

\begin{rk}
If $X_0 \in \HH^{\gamma}$, $\gamma\in (0,1)$, the
estimate \eqref{reg-x} also holds for  $\gamma\in (0,1)$. The key ingredient of the proof 
is use of the contractivity of $S(t)$ to deal with the term $\|S(t)X_0^N\|_{L^6}$. Indeed, \eqref{con-pro}  yields that  
\begin{align*}
\|S(t)X_0^N\|_{L^6}\le Ct^{-\frac 1{12}}\|X_0\|.
\end{align*}
From the  H\"older inequality,  it follows that there exist $p_1,q_1$ satisfying $\frac 1{p_1}+\frac 1{q_1}=1$, $(\frac {1}2+\frac {\gamma}4)p_1<1$ and $q_1< 4$, such that 
\begin{align*}
&\int_0^t(t-s)^{-\frac 12-\frac {\gamma}4}\|Y^N(s)\|_{L^6}^3ds
\le 
(\int_0^t(t-s)^{-(\frac {1}2+\frac {\gamma}4)p_1}ds)^{\frac 1{p_1}}(\int_0^t \|Y^N(s)\|_{L^6}^{3q_1}ds)^{\frac 1{q_1}}.
\end{align*}
Based on the above estimate and similar arguments in the proof of Proposition \ref{prop-spa}, we obtain the desired result.
\end{rk}

After these preparations, we are able to answer the well-posedness problem of Eq. \eqref{spde}. Before that, we give the useful lemma whose proof is similar to that of \cite[Lemma 4.3]{CHS19}.

\begin{lm}\label{sob}
Let $g: L^4\to H$ be the Nemytskii operator of a polynomial of second degree.  
Then for any $\beta\in (0,1)$, it holds that 
\begin{align*}
\|g(x)y\|_{\HH^{-1}}\le C\big(1+\|x\|_E^2+\|x\|_{\HH^{\beta}}^2\big)\|y\|_{\HH^{-\beta}},
\end{align*}
where $ x\in E,  x\in \HH^{\beta}$ and $y\in \HH.$
\end{lm}

\begin{prop}\label{pri-x-e}
Let $\sup\limits_{N\in \N^+}\|X^N_0\|_E\le C(X_0)$, $T>0$ and $q\ge1$. Then
the unique solution  $X^N$  of Eq. \eqref{sga}  satisfies 
\begin{align}\label{sup-xn}
\E\Big[\sup_{t\in [0,T]}\big\|X^N(t)\big\|_{E}^q\Big]
&\le C(X_0,T,q)
\end{align}
for a positive constant $ C(X_0,T,q)$.
\end{prop}

\begin{proof}
Due to Corollary \ref{cor-zn}, it remains to bound 
$\E\Big[\sup\limits_{t\in [0,T]}\big\|Y^N(t)\big\|_{E}^q\Big]$.
The mild form of $Y^N$, combined with \eqref{con-pro}, \eqref{smo-eff} and the estimation of $\|Y^N\|_{L^6}$,  yields that 
\begin{align*}
&\E\Big[\sup\limits_{t\in [0,T]}\|Y^N(t)\|_E^q\Big]\\
&\le \E\Big[\sup\limits_{t\in [0,T]}\|S(t)X^N_0\|_E^q\Big]
+C\E\Big[\Big(\int_0^T(t-s)^{-\frac 5{8}}\|F(Y^N+Z^N)\|ds\Big)^{q}\Big]\\
&\le C(X_0,T,q),
\end{align*}
which completes the proof.

\end{proof}

\section{Strong convergence analysis of the spectral Galerkin method}
\label{sec-str}

The main idea of our approach to proving the global existence of the solution is to show the uniform  convergence of the sequence  $\{(Y^N, Z^N)\}_{N \in \N^+}$ and then to prove the limit process is the unique mild solution of Eq. \eqref{spde}. 
In the following, we first present  the strong  convergence analysis of the spectral Galerkin approximation  in $\HH^{-1}$. We would like to mention that there already exists some convergence result of finite dimensional approximation for Eq. \eqref{spde} driving by additive space-time white noise (see e.g. \cite{CHS19}).
Different from the additive case, the convergence analysis of finite dimensional approximation for Eq. \eqref{spde} driving by multiplicative space-time noise is more involved and has not been studied yet.

\begin{prop}\label{str-con}
Let $ X_0\in \HH^{\gamma}$, $\gamma \in (0,\frac 32)$, $T>0$, $p\ge1$ and  $\sup\limits_{N\in \N^+}\|X^N_0\|_E<\infty$.
Assume that $X^N$ and $X^M$ are  the spectral Galerkin approximations with different parameters $N,M\in \N^+, N<M$. Then it holds that 
\begin{align}\label{sup-e-err}
&\sup_{t\in [0,T]}\E\Big[\|X^N(t)-X^M(t)\|_{\HH^{-1}}^{2p}\Big]
\le C(T,X_0,p) \lambda_{N}^{-\gamma p},
\end{align} 
where $C(T,X_0,p)$ is a positive constant.
\end{prop}

\begin{proof}
Due to Proposition \ref{prop-spa},
we obtain that for $t\in [0,T]$, $p\ge 1$ and  $\gamma\in (0,\frac 32)$,
\begin{align*}
\E\Big[\|(I-P^N)X^M(t)\|_{\HH^{-1}}^p\Big]
&\le \E\Big[\|(I-P^N)A^{-\frac 12-\frac \gamma 2}A^{\frac \gamma 2}X^M(t)\|^p\Big]\\
&\le C(X_0,T,p,\gamma)\lambda_N^{-\frac p2-\frac {\gamma p}2}.
\end{align*}
Thus it remains to estimate  $\|X^N-P^NX^M\|_{\HH^{-1}}$. 
From the Taylor expansion and It\^o formula, it follows that for $p\ge 2$,
\begin{align*}
&\|X^N(t)-P^NX^M(t)\|_{\HH^{-1}}^{2p}\\
&=-2p\int_0^t\|X^N(s)-P^NX^M(s)\|_{\HH^{-1}}^{2p-2}\|X^N(s)-P^NX^{M}(s)\|_{\HH^1}^2ds\\
&\quad-2p\int_0^t\|X^N(s)-P^NX^M(s)\|_{\HH^{-1}}^{2p-2} \Big\<\int_0^1 F'(\theta X^N(s)+(1-\theta)X^M(s))d\theta\\
&\qquad (X^N(s)-P^NX^M(s)-(I-P^N)X^M(s)),X^N(s)-P^NX^M(s)\Big\>ds\\
&\quad+2p\int_0^t\|X^N(s)-P^NX^M(s)\|_{\HH^{-1}}^{2p-2}\<X^N(s)-P^NX^M(s),\\
&\qquad (G(X^N(s))-G(X^M(s)))dW(s)\>_{\HH^{-1}}\\
&\quad+2p\int_0^t\sum_{i\in \N^+} \|X^N(s)-P^NX^M(s)\|_{\HH^{-1}}^{2p-2}\|P^N((G(X^N(s))-G(X^M(s)))e_i)\|^2_{\HH^{-1}}ds\\
&\quad+2p(2p-2)\int_0^t \|X^N(s)-P^NX^M(s)\|_{\HH^{-1}}^{2p-4}\sum_{i\in \N^+}|\<X^N(s)-P^NX^M(s),\\
&\qquad P^N(G(X^N(s))-G(X^M(s))e_i)\>_{\HH^{-1}}|^2ds\\
&=:-2p\int_0^t\|X^N(s)-P^NX^M(s)\|_{\HH^{-1}}^{2p-2}\|X^N(s)-P^NX^{M}(s)\|_{\HH^1}^2ds\\
&\quad+I_1(t)+I_2(t)+I_3(t)+I_4(t).
\end{align*}
The monotonicity of $-F$ and 
the Young inequality yield that
\begin{align*}
I_1&\le C\int_0^t\|X^N(s)-P^NX^M(s)\|_{\HH^{-1}}^{2p-2}\|X^N(s)-P^NX^M(s)\|^2ds\\
  &\quad+
2\int_0^t\|X^N(s)-P^NX^M(s)\|_{\HH^{-1}}^{2p-2}\Big\<A^{-\frac 12}
\int_0^1 F'(\theta X^N(s)
 +(1-\theta)X^M(s))d\theta\\
  &\qquad 
 (I-P^N)X^M(s)), 
 A^{\frac 12}(X^N(s)-P^NX^M(s))\Big\>ds.
 \end{align*}
From Lemma \ref{sob}, it follows that
for $\beta\in (0,1)$ and small $\epsilon>0$,
 \begin{align*}
 &I_1\le \epsilon \int_0^t\|X^N(s)-P^NX^M(s)\|_{\HH^{-1}}^{2p-2}\|X^N(s)-P^NX^M(s)\|^2_{\HH^1}ds\\
 &\quad
+C(\epsilon)\int_0^t\|X^N(s)-P^NX^M(s)\|_{\HH^{-1}}^{2p}ds\\
&\quad+
C(\epsilon)\int_0^t\Big \|X^N(s)-P^NX^M(s)\|_{\HH^{-1}}^{2p-2}\|A^{-\frac 12}\\
&\qquad \int_0^1 F'(\theta X^N(s)
 +(1-\theta)X^M(s))d\theta
 (I-P^N)X^M(s))\|^2ds\\
 &\le \epsilon \int_0^t\|X^N(s)-P^NX^M(s)\|_{\HH^{-1}}^{2p-2}\|X^N(s)-P^NX^M(s)\|^2_{\HH^1}ds\\
 &\quad
+C(\epsilon)\int_0^t\|X^N(s)-P^NX^M(s)\|_{\HH^{-1}}^{2p}ds\\
&\quad+
C(\epsilon)\lambda_N^{- \gamma  -\beta }\int_0^t\|X^N(s)-P^NX^M(s)\|_{\HH^{-1}}^{2p-2}\Big(1+\|X^N\|_{\HH^{\beta}}^4\\
&\qquad
+\|X^M\|_{\HH^{\beta}}^4+\|X^N\|_E^4+\|X^M\|_E^4\Big) \|X^M(s)\|_{\HH^{ \gamma }}^2ds.
\end{align*}
The uniform boundedness of $\{e_j\}_{j\in \N^+}$
and the Young inequality yield that for small $\epsilon>0$,
\begin{align*}
&\E\Big[I_3+I_4\Big]\\
&\le
C\E\Big[\int_0^t\|X^N(s)-P^NX^M(s)\|_{\HH^{-1}}^{2p-2}\sum_{i\in \N^+} \|P^N((G(X^N(s))-G(X^M(s)))e_i)\|^2_{\HH^{-1}}ds\Big]\\
&\le C\E\Big[\int_0^t \|X^N(s)-P^NX^M(s)\|_{\HH^{-1}}^{2p-2}\sum_{j\in \N^+} \|(G(X^N(s))-G(X^M(s)))e_j)\|^2\lambda_j^{-1}ds\Big]\\
&\le C(\epsilon)\E\Big[\int_0^t \|X^N(s)-P^NX^M(s)\|_{\HH^{-1}}^{2p}ds\Big]\\
&\quad+
\epsilon \E\Big[\int_0^t  \|X^N(s)-P^NX^M(s)\|_{\HH^{-1}}^{2p-2}\|X^N(s)-P^NX^M(s)\|_{\HH^{1}}^2ds\Big]\\
&\quad+ C\E\Big[\int_0^t  \|X^N(s)-P^NX^M(s)\|_{\HH^{-1}}^{2p-2}\|(I-P^N)X^M(s)\|^2ds\Big].
\end{align*}
The above estimations, combined with the Young inequality and  the martingale property of the stochastic integral $I_2$,  yield that  for $\beta\in (0,1)$ and small $\epsilon>0$, 
\begin{align*}
&\E\Big[\|X^N(t)-P^NX^M(t)\|_{\HH^{-1}}^{2p}\Big]\\
&\le -2p\int_0^t\E\Big[\|X^N(s)-P^NX^M(s)\|_{\HH^{-1}}^{2p-2}\|X^N(s)-P^NX^{M}(s)\|_{\HH^1}^2\Big]ds\\
&\quad+\E\Big[I_1(t)+I_2(t)+I_3(t)+I_4(t)\Big]\\
&\le 
C(\epsilon)\int_0^t \E\Big[\|X^N(s)-P^NX^M(s)\|_{\HH^{-1}}^{2p}\Big]ds\\
&\quad +
C(\epsilon)\lambda_N^{- \gamma  -\beta }\int_0^t
\E\Big[\|X^N(s)-P^NX^M(s)\|_{\HH^{-1}}^{2p-2}\Big(1+\|X^N\|_{\HH^{\beta}}^4\\
&\quad
+\|X^M\|_{\HH^{\beta}}^4+\|X^N\|_E^4+\|X^M\|_E^4\Big) \|X^M(s)\|_{\HH^{ \gamma }}^2\Big]ds\\
&\quad +
C\int_0^t \E\Big[ \|X^N(s)-P^NX^M(s)\|_{\HH^{-1}}^{2p-2}\|(I-P^N)X^M(s)\|^2\Big]ds\\
&\le C(\epsilon)\int_0^t \E\Big[\|X^N(s)-P^NX^M(s)\|_{\HH^{-1}}^{2p}\Big]ds+C(\epsilon)\lambda_N^{- \gamma p  -\beta p}\int_0^t\E\Big[\Big(1+\|X^N\|_{\HH^{\beta}}^4\\
&\quad
+\|X^M\|_{\HH^{\beta}}^4+\|X^N\|_E^4+\|X^M\|_E^4\Big)^{p} \|X^M(s)\|_{\HH^{ \gamma }}^{2p}\Big]ds\\
&\quad
+C\lambda_{N}^{-\gamma p}\int_0^t\E\Big[\|X^M(s)\|_{\HH^{\gamma}}^{2p}\Big]ds.
\end{align*}
Combining the regularity estimates of $X^N$ and $X^M$ in Proposition \ref{prop-spa}, we complete the proof by using the Gronwall inequality.
\end{proof}

Now, we are in the position to deduce the error estimate in $\HH$, which implies that $\{X^N\}_{N\in \N^+}$  is a Cauchy sequence in $L^{p}(\Omega; C([0,T];\HH))$. The following strong convergence rate of the spectral Galerkin approximation is also applied for analyzing the strong convergence of the full discretization and its density function in \cite{CH20}. 

\begin{tm}\label{str-err-h1}
Let $ X_0\in \HH^{\gamma}$,  $\gamma \in (0,\frac 32)$, $T>0$, $p\ge1$ and  $\sup\limits_{N\in \N^+}\|X^N_0\|_E<\infty$.
Assume that $X^N$ and $X^M$ are the spectral Galerkin approximations with different parameters $N,M\in \N^+, N<M$. Then for $\tau\in (0,\gamma)$, it holds that 
\begin{align}\label{sup-err-h1}
&\E\Big[\sup_{t\in [0,T]}\|X^N(t)-X^M(t)\|^{2p}\Big]
\le C(T,X_0,p) \lambda_{N}^{-\tau p}.
\end{align} 
for a positive constant $C(T,X_0,p)$.
\end{tm}
\begin{proof}
From the mild form of $X^N$ and $P^NX^M$, the smoothing effect \eqref{smo-eff} of $S(t)$, Lemma \ref{sob}, Proposition \ref{str-con}, the interpolation inequality and the Burkholder inequality,  it follows  that for $p\ge \max(\frac l2,2)$, $l>2$ and $\beta\in (0,1)$,
\begin{align*}
&\E\Big[\|X^N(t)-P^NX^M(t)\|^{2p}\Big]\\
&\le C(p,T)\big(\int_0^T(t-s)^{-\frac {l}{2(l-1)}}ds\big)^{\frac {2p(l-1)}l}\E\Big[\Big(\int_0^T \big(1+\|X^N\|_E^{2l}+\|X^M\|_E^{2l}+\|X^N\|_{\HH^{\beta}}^{2l}\\
&\quad +\|X^M\|_{\HH^{\beta}}^{2l}\big)\big\|X^N(s)-X^M(s)\big\|_{\HH^{-1}}^{\beta l}\big\|X^N(s)-X^M(s)\big\|^{(1-\beta)l}ds\Big)^{\frac {2p}{l}}
\Big]\\
&\quad+C(p,T)\E\Big[\big\|\int_0^t S(t-s)(G(X^N(s))-G(X^M(s)))dW(s)\big\|^{2p}\Big]\\
&\le C(p,T,X_0,\beta)\lambda_N^{-\beta p \gamma }+
C(p,T)\E\Big[\Big(\int_0^t\sum_{i\in \N^+}\|S(t-s)(G(X^N(s))\\
&\quad-G(X^M(s)))e_i\|^2ds\Big)^p\Big]\\
&\le C(p,T,X_0,\beta)\lambda_N^{-\beta p \gamma }+
C(p,T)\E\Big[\Big(\int_0^t\|X^N(s)-P^NX^M(s)\|^4ds\Big)^{\frac p2}\Big]\\
&\le C(p,T,X_0,\beta)\lambda_N^{-\beta p \gamma }+
C(p,T)\int_0^t\E\Big[\|X^N(s)-P^NX^M(s)\|^{2p}\Big]ds.
\end{align*}
From the Gronwall inequality, it follows that 
\begin{align}\label{sup-e-err1}
\sup_{t\in [0,T]}\E\Big[\|X^N(t)-P^NX^M(t)\|^{2p}\Big]
&\le C(X_0,T,p,\gamma)\lambda_N^{-\beta p \gamma }.
\end{align}

Furthermore, taking supreme over $t\in [0,T]$, similar arguments yield that for $p\ge \max(\frac l2,2)$, $l>2$ and $\beta\in (0,1)$,
\begin{align*}
&\E\Big[\sup_{t\in [0,T]}\|X^N(t)-P^NX^M(t)\|^{2p}\Big]\\
&\le 
C(p,T,X_0,\beta)\lambda_N^{-\beta p \gamma }\\
&
+C(p)\E\Big[\sup_{t\in [0,T]}\big\|\int_0^t S(t-s)(G(X^N(s))-G(X^M(s)))dW(s)\big\|^{2p}\Big].
\end{align*}
The factorization method yields that for $\alpha_1>\frac 1q$, $q>1$,
\begin{align*}
&\E\Big[\sup_{t\in [0,T]}\big\|\int_0^t S(t-s)(G(X^N(s))-G(X^M(s)))dW(s)\big\|^{2p}\Big]\\
&\le C(p,q,T)\E\Big[\|Z_{\alpha_1,N,M}\|_{L^q([0,T];  \HH)}^{2p}\Big],
\end{align*}
where $Z_{\alpha_1,N,M}(s)=\int_0^s(s-r)^{-\alpha_1}S(s-r)
P^N(G(X^N(r))-G(X^M(r)))dW(r)$.
Thus it suffices to estimate $\E\Big[\|Z_{\alpha_1,N,M}\|_{L^q([0,T];  H)}^{2p}\Big]$. 
From the H\"older and Burkholder inequalities,  Proposition \eqref{prop-spa} and \eqref{sup-e-err},  it follows that for ${2p}\ge q$,
\begin{align*}
&\E\Big[\|Z_{\alpha_1,N,M}\|_{L^q([0,T];  \HH)}^{2p}\Big]\\
&=
\E\Big[ \Big(\int_0^T\Big\|\int_0^s(s-r)^{-\alpha_1}S(s-r)
P^N(G(X^N(s))-G(X^M(s)))dW(r)\Big\|^qds\Big)^{\frac {2p}q}\Big]\\
&\le C(T,p)\int_0^T\E\Big[ \Big\|\int_0^s(s-r)^{-\alpha_1}S(s-r)
P^N(G(X^N(s))-G(X^M(s)))dW(r)\Big\|^{2p}\Big]ds\\
&\le C\int_0^T\E\Big[ \Big(\int_0^s(s-r)^{-2\alpha_1}\sum_{i\in \N^+}\Big\|S(s-r)
P^N(G(X^N(s))-G(X^M(s)))e_i\Big\|^2dr\Big)^{p}\Big]ds\\
&\le C\int_0^T\E\Big[ \Big(\int_0^s(s-r)^{-(2\alpha_1+\frac 14)}\|X^N(s)-P^NX^M(s)\|^2dr\Big)^{p}ds\Big]+C(T,p)\lambda_N^{-\gamma p}\\
&\le C\lambda_N^{-\gamma\beta p}.
\end{align*}
Combining the above estimates and $$\E\Big[\|\sup\limits_{t\in[0,T]}(I-P^N)X^M(t)\|^{2p}\Big]
\le C(T,X_0,p,\gamma) \lambda_{N}^{-\gamma p},$$ 
we complete the proof.

\end{proof}

\begin{rk}\label{rk-xn}
If  $X_0\in \HH^{\gamma}$, $\gamma>\frac 12$, then 
$\|X_0^N\|_E\le C(X_0)$ holds for every $N\in \N^+$.
If the bound of $\|X_0^N\|$ is not uniform, then by using \eqref{con-pro}, we have that
\begin{align*}
\E\big[\|X^N(t)\|^q_E\big]
&\le C(X_0,T,q)(1+t^{\min(-\frac 18-\epsilon+\frac \gamma 4,0)q}).
\end{align*}
As a result, it's is not hard to check that Proposition \ref{str-con}, Theorem \ref{str-err-h1} and Proposition \ref{well}
still hold with $p=1$, which is helpful for establishing the wellposedness result under mild assumptions.
\end{rk}

\section{Global existence and regularity estimate}
\label{sec-wel}

Based on the convergence of the approximate process $X^N$, we are in a position to show the global existence 
of the unique solution for Eq. \eqref{spde} driven by multiplicative space-time white noise.

\begin{prop}\label{well}
Let $T>0$, $ X_0\in \HH^{\gamma}$,  $\gamma \in (0,\frac 32)$, $p\ge1$ and  $\sup\limits_{N\in \N^+}\|X^N_0\|_E\le C(X_0)$. Then
Eq. \eqref{spde} possesses a unique mild solution
$X$ in $L^{2p}(\Omega; C(0,T; \HH))$.
\end{prop}

\begin{proof}
We first show the local uniqueness of the mild solution for Eq. \eqref{spde}. 
Let $\tau_R:=\inf\{ t\in  [0,T] \big| \|X(t)\|>R\}$. 
Then   
the uniqueness in $ [0,\tau_R]$ is obtained due to the Lipschitz continuity of $G$ and the local Lipschitz continuity of $F$. More precisely, assume that we have two different mild solutions $X_1$ and $X_2$ for Eq. \eqref{spde} with the same initial datum $X_0$.
Next, we prove the local uniqueness, i.e., in each $[0,\tau_R]$, $X_1=X_2$, a.s. 
Since the decompositions $X_1=Y_1+Z_1$ and $X_2=Y_2+Z_2$, we have for $t\in [0,\tau_R]$,
\begin{align*} 
d(Y_1(t)-Y_2 (t))
&=-A^2(Y_1(t)-Y_2 (t))dt
-A(F(X_1(t))-F(X_2(t)))dt,\\
Y_1(0)-Y_2 (0)&=0,
\end{align*}
and 
\begin{align*}
d(Z_1(t)-Z_2 (t))
&=-A^2(Z_1(t)-Z_2 (t))dt
+(G(X_1(t))-G(X_2(t)))dW(t),\\
Z_1(0)-Z_2(0)&=0.
\end{align*}
From the mild form of $Z_1-Z_2$ and the factorization method, it follows that 
\begin{align*}
&\E\Big[\sup_{t\in [0,\tau_R]}\|Z_1(t)-Z_2 (t)\|_E^p\Big]\\
&=\E\Big[\sup_{t\in [0,\tau_R]}\Big\|\int_0^t S(t-s) (G(X_1(s))-G(X_2(s)))dW(s)\Big\|^p\Big]\\
&\le C(q,T)
\E \Big[ \int_0^{\tau_R}\Big(\int_0^t (t-s)^{-2\alpha_1-\frac 14} \sum_{i\in \N^+}\|S(t-s) (G(X_1(s))\\
&\qquad -G(X_2(s)))e_i\|^2ds\Big)^{\frac p2}dt\Big]\\
&\le C(q,T) \E\Big[\sup_{t\in [0,\tau_R]}\big\| X_1(t)-X_2(t)\big\|^p\Big],
\end{align*}
where  $\alpha_1>\frac 1p+\frac 18$ for large enough  $p>1$.
Similar arguments, together with the Young and Gagliardo--Nirenberg inequality, yield that for $t\in [0,\tau_R]$ and for some $\epsilon<1$,
\begin{align*}
\|Y_1(t)-Y_2(t)\|^2
&\le -\int_0^t2\|A(Y_1(s)-Y_2(s))\|^2ds\\
&\quad+2\int_0^t\<-A(F(X_1(s))-F(X_2(s))), Y_1(t)-Y_2(t)\>ds\\
&\le C(\epsilon)\int_0^t \|F(X_1(s))-F(X_2(s))\|^2ds\\
&\le  C(\epsilon)\int_0^t \|X_1(s)-X_2(s)\|^2(1+\|X_1(s)\|_E^4+\|X_2(s)\|_{E}^4)ds\\
&\le C(\epsilon)\int_0^t \|X_1(s)-X_2(s)\|^2(1+
\|AY_1(s)\|^2+\|AY_2(s)\|^2+\|Y_1(s)\|^6\\
&\qquad+\|Y_2(s)\|^6
+\|Z_1(s)\|_E^4+\|Z_2(s)\|_E^4)ds.
\end{align*}
Notice that for $i=1,2$ and $t\in [0,\tau_R]$, 
\begin{align*}
\int_0^t\|AY_i(s)\|^2ds&\le C\|Y_0\|^2
+C\int_0^t(\|Z_i(s)\|_E^2+1)\|Y_i(s)\|_{L^4}^4ds\\
&\quad+  C\int_0^t
(1+\|\nabla Y_i(s)\|^2+\|Z_i(s)\|_{L^8}^8)ds\\
&\le C(R,T,Y_0)<\infty.
\end{align*}
By Gronwall's inequality, we get 
\begin{align*}
\|Y_1(t)-Y_2(t)\|^2
\le \exp(C(R, T, Y_0))\int_0^t\|X_1(s)-X_2(s)\|^2ds.
\end{align*}
From the previous estimates, we conclude that 
for $0\le s\le t\le \tau_R$ and $p\ge1$,
\begin{align*}
&\|X_1(s)-X_2(s)\|^{2p}\\
&\le 
C_p\|Y_1(s)-Y_2(s)\|^{2p}+C_p \|Z_1(s)-Z_2(s)\|^{2p}\\
&\le 
C_p\exp(C(R, T, Y_0))\int_0^{s}\|X_1(s)-X_2(s)\|^{2p}ds
+
C_p \|Z_1(s)-Z_2(s)\|^{2p},
\end{align*}
which, together with Gronwall's inequality, yields that 
\begin{align*}
\|X_1(s)-X_2(s)\|^{2p}
&\le \exp(C(p,T)\exp(C(R, T, Y_0)))\|Z_1(s)-Z_2(s)\|^{2p}. 
\end{align*}
Taking expectation and using the Burkholder inequality, we have  for large enough $q>1$,
\begin{align*}
&\E\Big[\sup_{s\in [0,t]}\|X_1(s)-X_2(s)\|^{2p}\Big]\\
&\le \exp(C(p,T)\exp(C(R, T, Y_0)))
\E\Big[\sup_{s\in [0,t]}\|Z_1(s)-Z_2(s)\|^{2p}\Big]\\
&\le  C(R,T,p,Y_0)\E \Big[\int_0^{t}\Big(\int_0^s (s-r)^{-\frac 14-\frac 2q} \| G(X_1(r))-G(X_2(r))\|^2dr\Big)^{p}dt\Big]\\
&\le C(R,T,p,Y_0)
\int_0^{t}  \E \Big[\sup_{r\in[0,s]}\| X_1(r))-X_2(r)\|^{2p}\Big]ds.
\end{align*}
From Gronwall's inequality, it follows that for any $t\le \tau_R$,
\begin{align*}
\E\Big[\sup_{s\in[0,t]}\|X_1(s)-X_2(s)\|^{2p}\Big]=0.
\end{align*}
Thus the local uniqueness of the mild solution holds. 
Once the global existence of the mild solution holds,
 we have 
\begin{align*}
\E\Big[\sup_{s\in[0,T]}\|X_1(s)-X_2(s)\|^{2p}\Big]&\le \lim_{R\to \infty}\E\Big[\sup_{s\in[0,\tau_R]}\|X_1(s)-X_2(s)\|^{2p}\Big]=0,
\end{align*}
since it holds that $\lim_{R\to \infty}\tau_R =T, a.s.$

In the following, we show the existence of the global mild solution.
According to Theorem \ref{str-err-h1}, we have 
that 
$\{X^N\}_{N\in \N^+}$ is a Cauchy sequence in $L^{2p}(\Omega; C([0,T];\HH))$. Then we denote $X$ the limit of $X^N$ in $L^{2p}(\Omega; C([0,T];\HH))$. From $\lim\limits_{N\to \infty}\|X^N-X\|_{L^{2p}(\Omega; C([0,T];\HH))}=0$, it follows that  
for each $i\in \N^+$,
\begin{align*}
\lim_{N\to \infty}\||\<X^N-X,e_i\>|\|_{L^{2p}(\Omega; C([0,T];\R))}=0,
\end{align*}
which implies that for a subsequence $\{X^{N_k}\}_{k\in \N^+}$,
\begin{align*}
\lim_{k\to \infty} \sup_{t\in[0,T]}|\<X^{N_k}(t),e_i\>|^2\lambda_i^{\gamma}
= \sup_{t\in[0,T]} |\<X(t),e_i\>|^2\lambda_i^{\gamma}, \; \text{a.s.},
\end{align*}
and $\lim_{N\to \infty} X^{N_k} = X$ in $C([0,T];\HH^{\gamma})$, a.s. 

The uniform boundedness of $\|X^N\|_{L^{2p}(\Omega;C([0,T];\HH^{\gamma}))}$, together with Fatou's lemma, yields that 
\begin{align*}
\|X\|^{2p}_{L^{2p}(\Omega;C([0,T];\HH^{\gamma}))}
&\le \liminf_{N\to \infty}\|X^{N_k}\|^{2p}_{L^{2p}(\Omega;C([0,T];\HH^{\gamma}))}\le C(T,X_0,p).
\end{align*}
Thus it suffices to prove that $X$ is the mild solution of Eq. \eqref{spde}, i.e.,
\begin{align*}
X(t)=S(t)X_0+\int_0^tS(t-s)F(X(s))ds
+\int_0^tS(t-s)G(X(s))dW(s),\; \text{a.s.}
\end{align*}
The mild form of $X^N$ and \eqref{smo-eff} yield that
\begin{align*}
&Err:=\|S(t)(I-P^N)X_0\|_{L^{2p}(\Omega; C([0,T;\HH]))}\\
&+\Big\|\int_0^tS(t-s)A(F(X(s))-P^NF(X^N(s)))ds\Big\|_{L^{2p}(\Omega; C([0,T];\HH))}\\
&+\Big\|\int_0^tS(t-s)(G(X(s))-P^NG(X^N(s)))dW(s)\Big\|_{L^{2p}(\Omega; C([0,T];\HH))}\\
&\le C(T,X_0)\lambda_N^{-\frac \gamma 2}+\Big\|\int_0^tS(t-s)A(I-P^N)F(X^N(s))ds\Big\|_{L^{2p}(\Omega; C([0,T];\HH))}\\
&+\Big\|\int_0^tS(t-s)AP^N(F(X(s)-F(X^N(s)))ds\Big\|_{L^{2p}(\Omega; C([0,T];\HH))}\\
&\quad+\Big\|\int_0^tS(t-s)(I-P^N)G(X(s)dW(s)\Big\|_{L^{2p}(\Omega; C([0,T];\HH))}\\
&+\Big\|\int_0^tS(t-s)P^N(G(X(s))-G(X^N(s)))dW(s)\Big\|_{L^{2p}(\Omega; C([0,T];\HH))}\\
&\le C(T,X_0,p)\lambda_N^{-\frac \gamma 2}+
C(T,p)\lambda_N^{-\frac \gamma 2}
\Big\|\int_0^t(t-s)^{-\frac 12-\frac \gamma 4}\|F(X^N(s))\|_{\HH}ds\Big\|_{L^{2p}(\Omega; C([0,T];\R))}\\
&+C(T,p)\Big\|\int_0^t (t-s)^{-\frac 12}(1+\|X(s)\|_E^2+\|X^N(s)\|_E^2)\\
&\qquad\|X(s)-X^N(s)\|ds\Big\|_{L^{2p}(\Omega; C([0,T];\R))}\\
&+C(T,p)\Big\|\int_0^t(t-s)^{-\alpha_1}S(t-s)(I-P^N)G(X(s)dW(s)\Big\|_{L^{2p}(\Omega; L^{q}([0,T];\HH))}\\
&+C(T,p)\Big\|\int_0^t(t-s)^{-\alpha_1}S(t-s)P^N(G(X(s))-G(X^N(s)))dW(s)\Big\|_{L^{2p}(\Omega; L^{q}([0,T];\HH))}
\end{align*}
According to the factorization method, the Burkholder inequality, \eqref{smo-eff} and the error estimate \eqref{sup-e-err1}, we have that 
for $ 2p\ge q$, $\alpha_1>\frac 1q$, sufficient large  $q>1$ and $\beta\in (0,1)$,

\begin{align*}
&Err
\le C(T,X_0,p)\lambda_N^{-\frac {\beta\gamma} 2}\\
&
+C(T,p)\Big\|\int_0^t(t-s)^{-\alpha_1}S(t-s)(I-P^N)G(X(s)dW(s)\Big\|_{L^{2p}(\Omega; L^{q}([0,T];\HH))}\\
&+C(T,p)\Big\|\int_0^t(t-s)^{-\alpha_1}S(t-s)P^N(G(X(s))-G(X^N(s)))dW(s)\Big\|_{L^{2p}(\Omega; L^{q}([0,T];\HH))}\\
&\le C(T,X_0,p)\lambda_N^{-\frac {\beta\gamma} 2}\\
&
+C(T,p)\lambda_N^{-\frac \gamma 2}\Big\|\int_0^t(t-s)^{-2\alpha_1-\frac 14-\frac \gamma 2}(1+\|X(s)\|^{2\alpha})ds\Big\|_{L^{2p}(\Omega; L^{q}([0,T];\R))}\\
&+C(T,p)\Big\|\int_0^t(t-s)^{-2\alpha_1-\frac 14}\|X(s)-X^N(s)\|ds\Big\|_{L^{2p}(\Omega; L^{q}([0,T];\R))}\\
&\le C(T,X_0,p)\lambda_N^{-\frac {\beta\gamma} 2}.
\end{align*}
The above estimation implies that 
$$S(t)X_0+\int_0^tS(t-s)F(X(s))ds
+\int_0^tS(t-s)G(X(s))dW(s)$$ is the limit of $X^N$
in $L^{2p}(\Omega; C([0,T];\HH))$. 

By the uniqueness of the limit in $L^{2p}(\Omega; C([0,T];\HH))$, we conclude that $X(t)=S(t)X_0+\int_0^tS(t-s)F(X(s))ds
+\int_0^tS(t-s)G(X(s))dW(s)$, a.s.

\end{proof}

From the arguments in the above proof, we immediately 
get that the following well-posedness result under mild assumptions.
As a cost, we can not obtain the optimal convergence rate of this Cauchy sequence $\{X^N\}_{N\in \mathbb N^+}$.

\begin{tm}\label{tm-well}
Let $T>0$, $ X_0\in \HH^{\gamma}, \gamma>0$, $p\ge1$. Then
Eq. \eqref{spde} possesses a unique mild solution
$X$ in $L^{2p}(\Omega; C([0,T]; \HH))$.
\end{tm}
\begin{proof}
Since the strong convergence in Theorem \ref{str-err-h1} holds with $p=1$ (see Remark \ref{rk-xn}),
we have that 
$\{X^N\}_{N\in \N^+}$ is a Cauchy sequence 
in $L^{2}(\Omega; C([0,T]; \HH))$, which implies that there exists a subsequence $\{X^{N_k}\}_{k\in \N^+}$  converging to $X$ in $C([0,T];\HH)$ a.s.
Notice that Lemma \ref{pri-xn} implies that $X^N\in L^{2p}(\Omega; C([0,T];\HH))$ for any $p\ge 1$.
By using the H\"older inequality and Fatou's lemma, we obtain
\begin{align*}
&\|X-X^N\|_{L^{2p}(\Omega; C([0,T];\HH))}\\
&\le \|X-X^N\|_{L^{4p-2}(\Omega; C([0,T];\HH))}
\|X-X^N\|_{L^{2}(\Omega; C([0,T];\HH))}\\
&\le \|X-X^N\|_{L^{2}(\Omega; C([0,T];\HH))}^{\frac 1{2p}}\big(C(X_0,T,p)+\lim_{k\to \infty}\|X^{N_k}\|_{L^{4p-2}(\Omega; C(0,T;\HH))}\big)^{\frac {2p-1}{2p}}\\
&\le  C(X_0,T,p)\|X-X^N\|_{L^{2}(\Omega; C([0,T];\HH))}^{\frac 1{2p}},
\end{align*}
which implies that $\{X^N\}_{N\in \N^+}$ is also a Cauchy sequence 
in $L^{2p}(\Omega; C([0,T]; \HH))$.
\end{proof}

\begin{rk}
Let $T>0$, $ X_0\in \HH$, $p\ge1$.
By the similar arguments in the proof of Proposition \ref{prop-tm} and Theorem \ref{tm-well},  one may  prove that 
Eq. \eqref{spde} possesses a unique mild solution
$X$ in $C([0,T];L^{2p}(\Omega; \HH))$.
\end{rk}

After establishing the well-posedness of Eq. \eqref{spde}, 
we turn to giving the following properties of the exact solution $X$. 
\begin{cor}
Let $ X_0\in \HH^{\gamma}$, $\gamma \in (0,\frac 32)$, $T>0$ and $p\ge1$.
The unique mild solution  $X$  of Eq. \eqref{spde}  satisfies 
\begin{align}\label{pri-x}
\E\Big[\sup_{t\in [0,T]}\big\|X(t)\big\|_{\HH^{\gamma}}^p\Big]
&\le C(X_0,T,p).
\end{align}
\end{cor}
\begin{proof}
By Proposition \ref{prop-spa}, Theorem \ref{str-err-h1}
and Fatou's Lemma,
we completes the proof. 
\end{proof}

\begin{prop}\label{x-e}
Let $ X_0\in E$, $T>0$ and $p\ge1$.
The unique mild solution  $X$  of Eq. \eqref{spde}  satisfies 
\begin{align*}
\sup_{t\in [0,T]}\E\Big[\big\|X(t)\big\|_{E}^p\Big]
&\le C(X_0,T,p).
\end{align*}
\end{prop}
\begin{proof}
The proof is similar to that of Proposition \ref{pri-x-e}.
\end{proof}

\begin{rk}\label{rk-conti}
Under the condition of Proposition \ref{x-e}, one can prove that 
the solution $X$ has almost surely continuous trajectories in $E$. In addition we assume that $X_0$ is $\beta$-H\"older continuous with $\beta\in (0,1)$. By using the fact that $S(\cdot)$ is an analytical semigroup in $E$ and similar arguments in the proof of Proposition \ref{prop-tm}, we have that $X$ is almost surely $\beta$-continuous in space and $\frac \beta 4$-continuous in time. 
\end{rk}

\begin{prop}\label{prop-tm}
Let $X_0\in \HH^{\gamma}$, $\gamma \in (0,\frac 32)$, $p\ge1$.
Then the unique mild solution  $X$  of Eq. \eqref{spde}  satisfies 
\begin{align}\label{reg-tm}
\|X(t)-X(s)\|_{L^p(\Omega;\HH)}
&\le C(X_0,T,p)(t-s)^{\frac \gamma  4}
\end{align}
for a positive constant $ C(X_0,T,p)$ and $0\le s\le t\le T$.
\end{prop}
\begin{proof}
From the mild form of $X$, it follows that 
\begin{align*}
\|X(t)-X(s))\|
&\le \left\|(S(t)-S(s))X_0\right\|
\\
&\quad+\int_0^s\Big\|(S(t-r)-S(s-r))AF(X(r))\Big\|dr
\\
&\quad+\int_s^t\Big\|S(t-r)AF(X(r))\Big\|dr\\
&\quad+\Big\|\int_0^s(S(t-r)-S(s-r))G(X(r))dW(r)\Big\|\\
&\quad+\Big\|\int_{s}^tS(t-r)G(X(r))dW(r)\Big\|.
\end{align*}
By taking $p$th moment, using  \eqref{pri-x} and the smoothing effect of $S(t)$, we get 
\begin{align*}
&\E\Big[\big\| (S(t)-S(s))X_0 \big\|^p\Big]\le C(T,X_0,p,\gamma)(t-s)^{\frac {\gamma p} 4},\\
&\E\Big[\int_0^s\big\|(S(t-r)-S(s-r))AF(X(r))\big\|^pdr\Big]\\
&\le C(T,p) \E\Big[\Big(\int_0^s (s-r)^{-\frac 12-\frac \gamma 4}\big\|(S(t-s)-I)A^{-\frac \gamma 2}\big\|
\big\|F(X(r))\big\|
dr \Big)^p\Big]\\
&\le C(T,X_0,p,\gamma)(t-s)^{\frac {\gamma p} 4},
\end{align*}
and 
\begin{align*}
\E\Big[\Big(\int_s^t\Big\|S(t-r)AF(X(r))\Big\|dr\Big)^p\Big]
&\le C(T,p)\E\Big[\Big(\int_s^t(t-r)^{-\frac 12}\|F(X(r))\|dr\Big)^p\Big]\\
&\le C(T,X_0,p)(t-s)^{\frac p2}.
\end{align*}
The Burkholder inequality and  \eqref{pri-x} yield that 
\begin{align*}
&\E\Big[\Big\|\int_0^s(S(t-r)-S(s-r))G(X(r))dW(r)\Big\|^p\Big]\\
&\le C(T,p)\E\Big[\Big(\int_0^s\sum_{i\in \N^+}\|S(s-r)(S(t-s)-I)G(X(r))e_i\|^2ds\Big)^{\frac p2}\Big]\\
&\le C(T,X_0,p,\gamma)(t-s)^{\frac {\gamma p} 4}.
\end{align*}
and 
\begin{align*}
&\E\Big[\Big\|\int_{s}^tS(t-r)G(X(r))dW(r)\Big\|^p\Big]\\
&\le C(T,p)\E\Big[\Big(\int_s^t\sum_{i\in \N^+}\|S(t-r)G(X(r))e_i\|^2ds\Big)^{\frac p2}\Big]\le C(T,p,X_0) (t-s)^{\frac {3p}8}.
\end{align*}
Combining all the above estimates, we complete the proof.
\end{proof}

\begin{rk}\label{con-xn}
Under the same condition as in Proposition \ref{prop-tm},
the solution of the spectral Galerkin method $X^N$ satisfies 
\begin{align*}
\|X^N(t)-X^N(s)\|_{L^p(\Omega;\HH)}
&\le C(X_0,T,p)(t-s)^{\frac \gamma  4},
\end{align*}
where $ C(X_0,T,p)>0$ and $0\le s\le t\le T$.
\end{rk}

As a result of Proposition \ref{well}, we have the following strong convergence rate of the spectral Galerkin method. 
\begin{cor}\label{strong-err-spa}
Let $X_0\in \HH^{\gamma}$, $\gamma \in (0,\frac 32)$, $T>0$, $p\ge1$ and $\sup\limits_{N\in \N^+}\|X^N_0\|_E\le C(X_0)$. Then for $\tau\in (0,\gamma)$,
there 
exists $ C(X_0,T,p)>0$ such that 
\begin{align}\label{strong-spa}
\big\|X^N-X\big\|_{L^p(\Omega;C([0,T];{\HH})}
&\le C(X_0,T,p)\lambda_N^{-\frac \tau 2}.
\end{align}
\end{cor}

As a consequence of the strong  convergence of 
the spectral Galerkin method, the following  exponential integrability property of the mild solution holds. We would like to mention that the exponential integrability property has many applications in non-global SPDE and its numerical approximation(see e.g. \cite{CHJ13,CHLZ17,CHS18b}).
\begin{cor} 
Let $X_0\in \HH$.
There exist $\beta>0$, $c>0$ such that for $t\in [0,T]$,
\begin{align*}
&\E \Big[\exp\Big(\frac 12e^{-\beta t}\|X(t)\|_{\HH^{-1}}^2+c\int_0^t e^{-\beta s}\|X(s)\|_{L^4}^4ds+c\int_0^te^{-\beta s}\|\nabla X(s)\|^2ds\Big)\Big]\\
&\le C(X_0,T).
\end{align*} 
\end{cor}
\begin{proof}

From  the Gagliardo--Nirenberg and Young inequalities,  it follows that
\begin{align*}
\int_0^t\|X^N-X\|^4_{L^4}ds
&\le  C\int_0^t\|\nabla(X^N-X)\|^2ds
+C\int_0^t\|X^N-X\|^6ds.
\end{align*}
The similar arguments in Proposition \ref{str-con} yield that for $t\in[0,T]$, 
$$\lim_{N\to \infty}\|X^N-X\|_{L^{2}(\Omega; L^{2}([0,t];\HH^1))}=0,$$ which together with the strong convergence of $X^N$ in $C([0,t];L^{6}(\Omega; \HH))$ implies that 
$\lim_{N\to \infty}\|X^N-X\|_{L^{4}(\Omega; L^{4}([0,t];L^4))}=0.$
Thus by Fatou's lemma, it suffices to show the uniform boundedness of the exponential moment for $X^N$.

Denote $\mu(x)=-A^2x-AP^NF(x)$ and  $\sigma(x)=P^NG(x)I_{\HH}$ and $U(x)=\frac 12\|x\|_{\HH^{-1}}^2$, where $x\in P^N(\HH)$.
By direct calculations and the interpolation inequality, 
 we get 
\begin{align*}
&\<DU(x),\mu(x)\>+\frac 12\text{tr}[D^2U(x)\sigma(x)\sigma^*(x)]+\frac 12 \|\sigma(x)^*DU(x)\|^2\\
&=\<x, -A^2x+AF(x)\>_{\HH^{-1}}
+\frac 12\sum_{i\in \N^+} \|P^N(G(x)e_i)\|_{\HH^{-1}}^2
+\frac 12 \sum_{i\in \N^+}\<x,G(x)e_i\>_{\HH^{-1}}^2\\
&\le -(1-\epsilon)\|\nabla x\|^2-(4c_4-\epsilon)\|x\|_{L^4}^4+\epsilon \|x\|_{\HH^{-1}}^2+C(\epsilon).
\end{align*}
Using the exponential integrability lemma in \cite{CHL16b} and taking $\beta=\epsilon$, we have 
\begin{align*}
&\E \Big[\exp\Big(e^{-\beta t}\frac 12\|X^N(t)\|_{\HH^{-1}}^2+(4c_4-\epsilon)\int_0^t e^{-\beta s}\|X^N(s)\|_{L^4}^4ds\\
&+(1-\epsilon)\int_0^te^{-\beta s}\|\nabla X^N(s)\|^2ds\Big)\Big]\le C(X_0,T,\epsilon),
\end{align*}
which, combined with Fatou's lemma, completes the proof.
\end{proof}

\section{Conclusion}
\label{sec-con}
In this paper, we introduce a new approach to studying the global existence and regularity estimate of the solution process for stochastic Cahn--Hilliard equation driven by multiplicative space-time white noise.  Compared to the existing work, we use the 
spectral Galerkin method, instead of the cut-off equation, to approximate the original equation. Then by proving the well-posedness and a priori estimates of the approximated equation, we show that the solution $\{X^N\}_{N\in\N^+}$  
 possesses the sharp strong convergence rate and thus is a Cauchy sequence in certain Banach space. As a consequence, the limit process of $X^N$ is shown to be the global solution of stochastic Cahn--Hilliard equation and to possess the optimal regularity estimates.  

\section{Acknowledgement}
The authors are very grateful to Professor Yaozhong Hu(University of Alberta) for his helpful discussions and suggestions.

\appendix

\bibliographystyle{amsplain}
\bibliography{bib}

\providecommand{\bysame}{\leavevmode\hbox to3em{\hrulefill}\thinspace}
\providecommand{\MR}{\relax\ifhmode\unskip\space\fi MR }
\providecommand{\MRhref}[2]{%
  \href{http://www.ams.org/mathscinet-getitem?mr=#1}{#2}
}
\providecommand{\href}[2]{#2}
\begin{thebibliography}{10}

\bibitem{ABK12}
D.~C. Antonopoulou, D.~Bl\"{o}mker, and G.~D. Karali, \emph{Front motion in the
  one-dimensional stochastic {C}ahn-{H}illiard equation}, SIAM J. Math. Anal.
  \textbf{44} (2012), no.~5, 3242--3280. \MR{3023410}

\bibitem{AKM16}
D.~C. Antonopoulou, G.~Karali, and A.~Millet, \emph{Existence and regularity of
  solution for a stochastic {C}ahn--{H}illiard/{A}llen--{C}ahn equation with
  unbounded noise diffusion}, J. Differential Equations \textbf{260} (2016),
  no.~3, 2383--2417. \MR{3427670}

\bibitem{CH58}
J.W. Cahn and J.E. Hilliard, \emph{Free energy for a nonuniform system {I}.
  {I}nterfacial free energy.}, J. Chem. Phys. \textbf{2} (1958), 258--267.

\bibitem{Car01}
C.~Cardon-Weber, \emph{Cahn-{H}illiard stochastic equation: existence of the
  solution and of its density}, Bernoulli \textbf{7} (2001), no.~5, 777--816.
  \MR{1867082}

\bibitem{Cer03}
S.~Cerrai, \emph{Stochastic reaction-diffusion systems with multiplicative
  noise and non-{L}ipschitz reaction term}, Probab. Theory Related Fields
  \textbf{125} (2003), no.~2, 271--304. \MR{1961346}

\bibitem{CHJ13}
S.~Cox, M.~Hutzenthaler, and A.~Jentzen, \emph{Local lipschitz continuity in
  the initial value and strong completeness for nonlinear stochastic
  differential equations}, arXiv:1309.5595.

\bibitem{CH20}
J.~Cui and J.~Hong, \emph{Absolute continuity and numerical approximation of
  stochastic {C}ahn--{H}illiard equation with unbounded noise diffusion},
  arXiv:1907.11869.

\bibitem{CHL16b}
J.~Cui, J.~Hong, and Z.~Liu, \emph{Strong convergence rate of finite difference
  approximations for stochastic cubic {S}chr\"odinger equations}, J.
  Differential Equations \textbf{263} (2017), no.~7, 3687--3713. \MR{3670034}

\bibitem{CHLZ17}
J.~Cui, J.~Hong, Z.~Liu, and W.~Zhou, \emph{Strong convergence rate of
  splitting schemes for stochastic nonlinear {S}chr\"{o}dinger equations}, J.
  Differential Equations \textbf{266} (2019), no.~9, 5625--5663. \MR{3912762}

\bibitem{CHS19}
J.~Cui, J.~Hong, and L.~Sun, \emph{Numerical analysis of a full discretization
  for stochastic {C}ahn--{H}illiard equation driven by additive noise},
  arXiv:1812.06289.

\bibitem{CHS18b}
\bysame, \emph{On global existence and blow-up for damped stochastic nonlinear
  {S}chr\"{o}dinger equation}, Discrete Contin. Dyn. Syst. Ser. B \textbf{24}
  (2019), no.~12, 6837--6854. \MR{4026906}

\bibitem{DD96}
G.~Da~Prato and A.~Debussche, \emph{Stochastic {C}ahn-{H}illiard equation},
  Nonlinear Anal. \textbf{26} (1996), no.~2, 241--263. \MR{1359472}

\bibitem{Dap92}
G.~Da~Prato and J.~Zabczyk, \emph{Stochastic equations in infinite dimensions},
  Encyclopedia of Mathematics and its Applications, vol.~44, Cambridge
  University Press, Cambridge, 1992. \MR{1207136}

\bibitem{DG11}
A.~Debussche and L.~Gouden\`ege, \emph{Stochastic {C}ahn-{H}illiard equation
  with double singular nonlinearities and two reflections}, SIAM J. Math. Anal.
  \textbf{43} (2011), no.~3, 1473--1494. \MR{2821592}

\bibitem{DN91}
Q.~Du and R.~A. Nicolaides, \emph{Numerical analysis of a continuum model of
  phase transition}, SIAM J. Numer. Anal. \textbf{28} (1991), no.~5,
  1310--1322. \MR{1119272}

\bibitem{EN00}
K.~J. Engel and R.~Nagel, \emph{One-parameter semigroups for linear evolution
  equations}, Graduate Texts in Mathematics, vol. 194, Springer-Verlag, New
  York, 2000, With contributions by S. Brendle, M. Campiti, T. Hahn, G.
  Metafune, G. Nickel, D. Pallara, C. Perazzoli, A. Rhandi, S. Romanelli and R.
  Schnaubelt. \MR{1721989}

\bibitem{NS84}
A.~Novick-Cohen and L.~A. Segel, \emph{Nonlinear aspects of the
  {C}ahn-{H}illiard equation}, Phys. D \textbf{10} (1984), no.~3, 277--298.
  \MR{763473}

\bibitem{PZ07}
S.~Peszat and J.~Zabczyk, \emph{Stochastic partial differential equations with
  {L}\'{e}vy noise}, Encyclopedia of Mathematics and its Applications, vol.
  113, Cambridge University Press, Cambridge, 2007, An evolution equation
  approach. \MR{2356959}

\bibitem{PR07}
C.~Pr\'ev\^ot and M.~R\"ockner, \emph{A concise course on stochastic partial
  differential equations}, Lecture Notes in Mathematics, vol. 1905, Springer,
  Berlin, 2007. \MR{2329435}

\bibitem{VVW08}
J.~van Neerven, M.~C. Veraar, and L.~Weis, \emph{Stochastic evolution equations
  in {UMD} {B}anach spaces}, J. Funct. Anal. \textbf{255} (2008), no.~4,
  940--993. \MR{2433958}

\end{thebibliography}
\end{document}